\documentclass[12pt]{article}
\usepackage[centertags]{amsmath}
\usepackage{amsfonts}
\usepackage{amssymb}
\usepackage{amsthm}
\input{amssym.def}
\usepackage{graphicx}

\numberwithin{equation}{section}

\newtheorem{Definition}{Definition}[section]
\newtheorem{definition}[Definition]{Definition}
\newtheorem{theorem}[Definition]{Theorem}

\newtheorem{proposition}[Definition]{Proposition}

\newcommand{\mbn}{\mathbb{N}}

\begin{document}
\title{\Large \bf On the power graph of the direct product of two groups}
\author{A. K. Bhuniya and Sajal Kumar Mukherjee}
\date{}
\maketitle

\begin{center}
Department of Mathematics, Visva-Bharati, Santiniketan-731235, India. \\
anjankbhuniya@gmail.com, shyamal.sajalmukherjee@gmail.com
\end{center}

\begin{abstract}
The power graph $P(G)$ of a finite group $G$ is the graph with vertex set $G$ and two distinct vertices are adjacent if either of them is a power of the other. Here we show that the power graph $P(G_1 \times G_2)$ of the direct product of two groups $G_1$ and $G_2$ is not isomorphic to either of the direct, cartesian and normal product of their power graphs $P(G_1)$ and $P(G_2)$. A new product of graphs, namely generalized product, has been introduced and we prove that the power graph $P(G_1 \times G_2)$ is isomorphic to a generalized product of $P(G_1)$ and $P(G_2)$.
\end{abstract}


\section{Introduction}
The directed power graph of a semigroup was defined by Kelarev and Quinn \cite{K}. Then Chakraborty
et. al \cite{CGS} defined the undirected power graph $P(S)$ of a semigroup $S$ as the graph with vertex set $S$ and two distinct vertices $a$ and $b$ are adjacent if either $a^m=b$ or $b^n=a$ for some $m, n \in \mbn$. There are many articles associating group theoretic behavior of $G$ and the graph theoretic properties of $P(G)$. We refer to the survey \cite{survey} for an account of the development on the power graph of groups and semigroups. Success attained by the researchers in this direction influenced others to generalize power graph of groups to strong power graph \cite{cb}, deleted power graph $P^*(G)$ \cite{MRS}, enhanced power graph \cite{Aalipour}.

Here we investigate relationship of $P(G_1 \times G_2)$ with $P(G_1)$ and $P(G_2)$. There are many standard products of graphs, already defined, namely, direct product, cartesian product, normal product  etc. Here we show that, in general, $P(G_1\times G_2 )$ is not isomorphic to either of these products of $P(G_1)$ and $P(G_2)$. So we introduce a new product of two graphs $\Gamma_1$ and $\Gamma_2$, which we call generalized product of $\Gamma_1$ and $\Gamma_2$. Reason behind such naming is that each of the cartesian, direct and normal products is a special case of the generalized product of two graphs.

Our main theorem states that $P(G_1 \times G_2)$ is isomorphic to a generalized product of the power graphs $P(G_1)$ and $P(G_2)$.

\section{Main result}
We refer to \cite{Godsil} and \cite{w} for the notions on graph theory and to \cite{Hunger} for group theoretic background.

Throughout this article $\mathbb{N}$ denotes the set of natural  numbers and $\mathbb{Z}^\sharp=\mathbb{N} \bigcup \{0\}$.

Let us recall different standard notions of products of graphs.
\begin{definition}
Let $\Gamma_1$ and $\Gamma_2$ be two graphs.
\begin{itemize}
\item[(i)]
The direct product $\Gamma_1 \times \Gamma_2$ of $\Gamma_1$ and $\Gamma_2$ is defined as follows:
\begin{center}
 $V(\Gamma_1 \times \Gamma_2)= V(\Gamma_1) \times V(\Gamma_2)$ and $(g_1, g_2) \sim (\acute{g_1},\acute{g_2})$ if and only if $ g_1 \sim \acute{g_1}$ and $g_2 \sim \acute{g_2}$.
\end{center}
\item[(ii)]
The cartesian product $\Gamma_1 \boxtimes \Gamma_2$ of $\Gamma_1$ and $\Gamma_2$ is defined as follows:
\begin{center}
 $V(\Gamma_1\boxtimes \Gamma_2) = V(\Gamma_1)\times V(\Gamma_2)$ and $(g_1,g_2)\sim (\acute{g_1},\acute{g_2})$ if and only if either $g_1 = \acute{g_1}$ and $g_2 \sim \acute{g_2}$ or $g_1 \sim \acute{g_1}$ and $g_2 = \acute{g_2}$.
\end{center}
\item[(iii)]
The normal product $\Gamma_1 \ast \Gamma_2$ of $\Gamma_1$ and $\Gamma_2$ is defined as follows:
\begin{center}
 $V(\Gamma_1 \ast \Gamma_2) = V(\Gamma_1) \times V(\Gamma_2)$ and $(g_1,g_2)\sim (\acute{g_1},\acute{g_2})$ if and only if either $ g_1 \sim \acute{g_1}$ and $g_2 \sim \acute{g_2}$ or $g_1 = \acute{g_1}$ and $g_2 \sim \acute{g_2}$ or     $ g_1 \sim \acute{g_1}$ and $g_2 = \acute{g_2}$.
\end{center}
\end{itemize}
\end{definition}

Now we show that neither of these notions of product of graphs is enough to catch the relationship of $P(G_1 \times G_2)$ with $P(G_1)$ and $P(G_2)$.
\begin{proposition}
Let $G_1$ and $G_2$ be two nontrivial finite groups. Then $P(G_1\times G_2)$ is not isomorphic to $P(G_1) \boxtimes P(G_2)$.
\end{proposition}
\begin{proof}
Suppose, on the contrary that the two graphs, stated in the theorem are isomorphic. Let $e_i$ be the identity element of the group $G_i$. Let $g_1, g_2$ be two nonidentity elements of $G_1$ and $G_2$ respectively. Then $(e_1,e_2)$ is not adjacent to $(g_1, g_2)$ in $P(G_1)\boxtimes P(G_2)$, whereas $(e_{G_1}, e_{G_2}) \sim (g_1, g_2)$ in $P(G_1\times G_2)$. A contradiction.
\end{proof}

Consider $G_1 = G_2 = \mathbb{Z}_2$. Then $P(\mathbb{Z}_2 \times \mathbb{Z}_2)$ has precisely three edges, each edge emanating from the identity of $ \mathbb{Z}_2\times \mathbb{Z}_2$  and connects the remaining three vertices, whereas $P(\mathbb{Z}_2)\ast P(\mathbb{Z}_2)$ is the complete graph $K_4$ and $P(\mathbb{Z}_2)\times P(\mathbb{Z}_2)$ is a graph with precisely two edges.

Thus we show that $P(G_1\times G_2)$ is neither isomorphic to the direct product nor to the normal product of $P(G_1)$ and $P(G_2)$, in general.

So we need to introduce another product of graphs to uncover the relationship of the power graphs $P(G_1)$ and $P(G_2)$ with $P(G_1 \times G_2)$ for any two groups $G_1$ and $G_2$.

Before going into technical details, note the following fact. Let $G$ be a finite group and $a, b \in G $. Suppose that $a \sim b$ in $P(G)$. It is easy to observe that if $n$ is the smallest positive integer such that $a^n = b$, then $\{m \in \mathbb{N} : a^m = b\}$ is the arithmetic progression with initial term $n$ and common difference $o(a)$.

For any two integers $a$ and $d$, we denote the arithmetic progression with initial term $a$ and common difference $d$ by $AP(a, d)$.

Let $\Gamma$ be a graph. Then by a generalization on $\Gamma$ we mean a function $W: A(\Gamma)\bigcup \triangle \rightarrow \mathbb{Z}^\sharp \times  \mathbb{Z}^\sharp  $, where $A(\Gamma)$ is the arc set of $\Gamma$ and $\triangle = \{ (v, v): v \; \textrm{is a vertex of} \; \Gamma\}$.
\begin{definition}
Let $(\Gamma_1,W_1)$ and $(\Gamma_2,W_2)$ be two graphs equipped with two generalizations $W_1$, $W_2$  respectively. Then the generalized product $\Gamma_1 \ _{W_1}\times_{W_2} \Gamma_2$ is a graph with vertex set $V(\Gamma_1) \times V(\Gamma_2)$ and $(g_1,g_2) \sim (\acute{g_1},\acute{g_2})$ if and only if the following two conditions hold simultaneously:
\begin{enumerate}
\item[(i)]
$(g_1,g_2)\neq (\acute{g_1},\acute{g_2})$ and
\item[(ii)]
$AP(W_1(g_1,\acute{g_1})) \cap AP(W_2(g_2,\acute{g_2})) \cap \mathbb{N} \neq \phi$  or  $AP(W_1(\acute{g_1},g_1)) \cap AP(W_2(\acute{g_2},g_2)) \cap \mathbb{N} \neq \phi$.
\end{enumerate}
\end{definition}

If no question of ambiguity arise, then we denote a generalized direct product of two graphs $\Gamma_1$ and $\Gamma_2$ by $\Gamma_1 \times_W \Gamma_2$.

The following result justifies the name `generalized product'.

\begin{theorem}
Each of the direct, cartesian and normal products is a generalized product.
\end{theorem}
\begin{proof}
Let $\Gamma_1$ and $\Gamma_2$ be two graphs.
\begin{enumerate}
\item[(i)]
Consider the generalizations $W_1$ of $\Gamma_1$ defined by:
\begin{align*}
W_1(x,y) &= (1,1) \; \textrm{if} \; x \neq y \\
         &= (0,0) \; \textrm{if} \; x = y.
\end{align*}
and $W_2$ of $\Gamma_2$ defined similarly. Then it is easy to verify that $\Gamma_1 \times_W \Gamma_2 = \Gamma_1 \times \Gamma_2$.
\item[(ii)]
Consider the generalizations $W_1$ of $\Gamma_1$ defined by:
\begin{align*}
W_1(x,y) &= (1, 0) \; \textrm{if} \; x \neq y \\
         &= (1, 1) \; \textrm{if} \; x = y.
\end{align*}
and $W_2$ of $\Gamma_2$ defined by:
\begin{align*}
W_1(x,y) &= (2, 0) \; \textrm{if} \; x \neq y \\
         &= (1, 1) \; \textrm{if} \; x = y.
\end{align*}
Then $\Gamma_1 \times_W \Gamma_2 = \Gamma_1 \boxtimes \Gamma_2$.
\item[(iii)]
Consider the generalizations $W_1$ of $\Gamma_1$ defined by:
\begin{align*}
W_1(x,y) &= (1,0) \; \textrm{if} \; x \neq y \\
         &= (1, 1) \; \textrm{if} \; x = y.
\end{align*}
and $W_2$ of $\Gamma_2$ defined similarly. Then $\Gamma_1 \times_W \Gamma_2 = \Gamma_1 \ast \Gamma_2$.
\end{enumerate}
\end{proof}

Now, we prove our main theorem.
\begin{theorem}
For two groups $G_1$ and $G_2$, $P(G_1 \times G_2)$ and $P(G_1) \times_W P(G_2)$ are isomorphic for some choice of generalizations $W_1$ and $W_2$ of $P(G_1)$ and $P(G_2)$ respectively.
\end{theorem}
\begin{proof}
Let us first specify the choice of the generalizations of $P(G_1)$ and $P(G_2)$. We consider the generalization $W_1$ of $P(G_1)$ defined by:
\begin{align*}
W_1(a, b) &= (t, o(a)) \; \textrm{if $t$ is the smallest positive integer such that $a^t = b$} \\
         &= (0, 0) \; \textrm{otherwise}.
\end{align*}
and the generalization $W_2$ of $P(G_2)$ defined similarly. 

Let $(g_1,g_2) \sim (\acute{g_1},\acute{g_2}) $ in $P(G_1 \times G_2) $. Then either $(g_1,g_2)^m = (\acute{g_1},\acute{g_2}) $ or $(\acute{g_1},\acute{g_2})^n = (g_1,g_2)$ for some $m, n \in \mbn$. If $(g_1,g_2)^m = (\acute{g_1},\acute{g_2})$ then $g_1^m = \acute{g_1}$ implies that $m \in AP(t, o(g_1))$ where $t$ is the smallest positive integer for which $g_1^t = \acute{g_1}$. Hence $m \in AP(W_1(g_1,\acute{g_1}))$. Similarly $g_2^m = \acute{g_2}$ implies that $m \in AP(W_2(g_2,\acute{g_2}))$. Thus $AP(W_1(g_1,\acute{g_1})) \cap AP(W_2(g_2,\acute{g_2})) \cap \mathbb{N} \neq \phi $ and so $(g_1,g_2) \sim (\acute{g_1},\acute{g_2})$ in $ P(G_1)\times_W P(G_2)$. If $(\acute{g_1},\acute{g_2})^n = (g_1,g_2)$ then $AP(W_1(\acute{g_1}), g_1) \cap AP(W_2(\acute{g_2}), g_2) \cap \mathbb{N} \neq \phi $ and so $(g_1,g_2) \sim (\acute{g_1},\acute{g_2})$ in $ P(G_1)\times_W P(G_2)$.

Conversely, let $(g_1,g_2) \sim (\acute{g_1},\acute{g_2})$ in $ P(G_1)\times_W P(G_2)$. If $m \in AP(W_1(g_1,\acute{g_1})) \cap AP(W_2(g_2,\acute{g_2})) \cap \mathbb{N} \neq \phi$, then $(g_1,g_2)^m = (\acute{g_1},\acute{g_2}) $  and hence $(g_1,g_2) \sim (\acute{g_1},\acute{g_2})$ in  $P(G_1 \times G_2)$. Similar is the other case.
\end{proof}

\noindent
\textbf{Acknowledgement:} The second author is partially supported by CSIR-JRF grant.

\bibliographystyle{amsplain}

\end{document}